\documentclass[11pt]{article}
\usepackage{amsmath,amsthm,amsfonts,amssymb,amscd,enumerate}
\RequirePackage{graphics}

 \theoremstyle{plain} \newtheorem{theorem}{Theorem}[section]
\newtheorem{proposition}[theorem]{Proposition}
\newtheorem{corollary}[theorem]{Corollary}
\newtheorem{lemma}[theorem]{Lemma}

\newtheorem*{question}{Question}

 \theoremstyle{definition}
\newtheorem{definition}[theorem]{Definition} \theoremstyle{remark}
\newtheorem{remark}[theorem]{Remark}
\newtheorem{Remarks}[theorem]{Remarks}
\newtheorem*{claim}{Claim}

\newtheorem{example}[theorem]{Example}

\newcommand{\cc}{{\mathbb C}} 
 
\newcommand{\rr}{{\mathbb R}} 
\newcommand{\zz}{{\mathbb Z}}

\newcommand{\minus}{{-1}}
\newcommand{\ol}{\overline} \newcommand{\wt}{\widetilde}

\newcommand{\ga}{\alpha} 
\newcommand{\gd}{\delta} 
\newcommand{\gf}{\varphi}
\newcommand{\gi}{\iota} 
\newcommand{\gs}{\sigma} \newcommand{\gt}{\tau}
 
\newcommand{\gl}{\lambda}

\newcommand{\gG}{\Gamma} 
\newcommand{\gGen}{\Gamma_{n+1}}
\newcommand{\gGek}{\Gamma_{k}}
\newcommand{\gGekko}{\Gamma_{k+1}}
 \newcommand{\gL}{\Lambda}
\newcommand{\gS}{\Sigma}
\newcommand{\gT}{\Theta}
\newcommand{\gTh}{\Theta^H_3}

\newcommand{\bo}{\partial}
\newcommand{\bohat}{\hat{\partial}}

\newcommand{\none}{_{n+1}}
\newcommand{\kone}{_{k+1}}

\newcommand{\id}{\operatorname{id}}

\newcommand{\diff}{\operatorname{Diff}}
\newcommand{\Hom}{\operatorname{Hom}}
\newcommand{\Coker}{\operatorname{Coker}}
\newcommand{\Ker}{\operatorname{Ker}}

\newcommand{\GHS}{\operatorname{GHS}}
\newcommand{\HM}{\operatorname{HM}}

\newcommand{\cone}{\operatorname{Cone}}

\newcommand{\prin}{{\mathrm{prin}}}

\newcommand{\ue}{\underline{E}}

\newcommand\mapright[1]{\,\smash{\mathop{\longrightarrow\,}\limits^{#1}}}

\newenvironment{enumerate1}{
\begin{enumerate}[\upshape (1)]}%
    {
\end{enumerate}
}
\newenvironment{enumerate1*}{
\begin{enumerate}[\upshape (*1)]}%
	{
\end{enumerate}
}
	{
\end{enumerate}
}
\newcommand{\comment}[1]{}

\usepackage{color}
\usepackage[outerbars,color]{changebar}
\usepackage{ifthen}

\newboolean{usecolor}
\setboolean{usecolor}{true}

\ifthenelse{\boolean{usecolor}}
{
  \definecolor{colore}{cmyk}{0,1,0.6,0}
  \definecolor{coloregen}{cmyk}{0.7,0,1,0}
  \definecolor{coloresimo}{cmyk}{1,0.6,0,0}
}
{
  \definecolor{colore}{cmyk}{0,0,0,1}
  \definecolor{coloregen}{cmyk}{0,0,0,1}
  \definecolor{coloresimo}{cmyk}{0,0,0,1}
}

\newenvironment{enumeratei}{\begin{enumerate}[\upshape (i)]}%
		{\end{enumerate}}
\newenvironment{enumerateI}{\begin{enumerate}[\upshape (I)]}{\end{enumerate}}

\numberwithin{equation}{section} 

\begin{document}

\title{When are two Coxeter orbifolds diffeomorphic?}

\author{Michael W. Davis \thanks{I thank the Mathematical Sciences Center at Tsinghua University for the opportunity to visit there during the spring of 2013.  I began writing this paper during that period.  I also was partially supported by
NSF grant DMS 1007068.}    
}
\date{\today} \maketitle

\begin{abstract} One can define what it means for a compact manifold with corners to be a ``contractible manifold with contractible faces.''
Two combinatorially equivalent, contractible manifolds with contractible faces are diffeomorphic if and only if their $4$-dimensional faces are diffeomorphic.  It follows that two simple convex polytopes are combinatorially equivalent if and only if they are diffeomorphic as manifolds with corners.  On the other hand, by a result of Akbulut,  for each $n\ge 4$, there are smooth, contractible $n$-manifolds with contractible faces which are combinatorially equivalent but not diffeomorphic.  Applications are given to rigidity questions for reflection groups and smooth torus actions.
\smallskip

\noindent
\textbf{AMS classification numbers}.  Primary: 57R18, 57R55, 20F55, 20F65\\
Secondary: 57S15, 57R91
\smallskip

\end{abstract}

\section{Introduction}
More than once  
during the past few years\footnote{Taras Panov and Mikiya Masuda asked me this at a 2010 conference on toric topology in Banff.  More recently, Suhyoung Choi asked me the same question in connection with projective representations of Coxeter groups, cf.~\cite[Question~3.3]{choi}.} 
I have been asked,  ``Are combinatorially equivalent Coxeter orbifolds diffeomorphic?''   By ``Coxeter orbifold'' the questioner means something like the fundamental polytope of a geometric reflection group.  The underlying space of such an orbifold has the structure of a manifold with corners.  Usually, the questioner also wants to require that the underlying space is diffeomorphic, as a manfold with corners, to a simple convex polytope.  In this context the answer to the question is affirmative although the proof is not obvious (see Wiemeler \cite[Corollary~5.3]{wiemeler} and Corollary~\ref{cor:diffIII}, below).  However, if one weakens the definition by only requiring the strata to be compact contractible manifolds, there are counterexamples (cf.~Theorem~\ref{t:notdiff}).  The problem is caused by the $4$-dimensional strata.  

An orbifold $Q$ is  \emph{reflection type} if its local models are finite linear reflection groups on $\rr^n$. 
(These were called ``reflectofolds'' in \cite{d10}.)  Since the orbit space of a finite linear reflection group is the product of a Euclidean space with a simplicial cone,  a smooth orbifold of reflection type naturally has the structure of a smooth manifold with corners.  One can label each codimension two stratum by an integer $m\ge 2$ to indicate that the local dihedral group along the stratum has order $2m$.    So, $Q$ is a smooth manifold with corners and the labeling is  such that it determines a \emph{finite} Coxeter group of rank $k$ on each codimension $k$ stratum.  Conversely, if $Q$ is a smooth manifold with corners with such a labeling on its codimension two strata, then it can be given the structure of a smooth orbifold of reflection type and this structure is unique up to isotopy (cf.~\cite[Section 17]{d83}) .  There are several possibilities for when $Q$ should be called a ``Coxeter orbifold.''  Here are three successively stronger definitions:
\begin{enumerateI}
\item
Each stratum of $Q$ is a compact contractible manifold.
\item
Each stratum of $Q$ is homeomorphic to a disk.  
\item
$Q$ is isomorphic (as a manifold with corners) to a simple convex polytope.
\end{enumerateI}
Any $Q$ satisfying (I) is a \emph{Coxeter orbifold }.  It is  \emph{type}  (II) or (III) if it satisfies the corresponding condition.
When  $Q$ is smooth, note that  if a stratum is homeomorphic to $D^k$, then, for $k\neq 4$, it is diffeomorphic to $D^k$, but, for $k=4$, this implication is equivalent to the smooth $4$-dimensional Poincar\'e Conjecture.  

Under any of the above definitions, there is a  simplicial complex dual to $Q$, called its \emph{nerve} and denoted $N(Q)$.  
The vertices of $N(Q)$ are the codimension one strata of $Q$.  A set of vertices spans a simplex of $N(Q)$ if and only if the intersection of the corresponding set of codimension one faces is nonempty.  For $Q$ of type (III), $N(Q)$ is the boundary complex of the simplicial polytope that is dual to $Q$.  For type (II), $N(Q)$ is a triangulation of $S^{n-1}$ such that the link of each $k$-simplex of $N(Q)$ is homeomorphic to $S^{n-k-2}$.  For a general $Q$ of type (I), $N(Q)$ is a \emph{generalized homology $(n-1)$-sphere}, (or a $\GHS^{n-1}$ for short), where this means  a polyhedral homology $(n-1)$-manifold which has the same homology as $S^{n-1}$ (cf.~\cite[p.\,192]{dbook}).  The labels on the codimension two strata of $Q$ become labels on the edges of $N(Q)$.  A \emph{combinatorial equivalence} between Coxeter orbifolds $Q$ and $Q'$ is a label-preserving simplicial isomorphism $N(Q)\to N(Q')$.

The $1$-skeleton of a simplicial complex $N$ together with a labeling of its edges by integers $\ge 2$ determines a Coxeter system $(W,S)$ (cf.~\cite[Ex.~7.1.6]{dbook}).   The set of generators $S$ can be identified with the vertex set of $N$.  The labeling is \emph{proper} if for each simplex $\gs$ of $N$ the special subgroup $W_{S(\gs)}$, generated by the vertices of $\gs$, is finite.
First, there is the question of existence.  If $N$ is a $\GHS^{n-1}$ with a proper labeling of its edges, is there a Coxeter orbifold $Q$ with $N(Q)= N$?  This was addressed in \cite{d83} where the following result was proved.  

\begin{theorem}\label{t:annals}\textup{(\cite[Theorems  12.2 and 17.1]{d83}).}  Suppose $N$ is a $GHS^{n-1}$ with a proper labeling of its edges.  Then there is topological Coxeter orbifold $Q$ with $N(Q)=N$.  Noting that each $3$-dimensional link in $N$ is a PL homology $3$-sphere (and therefore, has a unique smooth structure), we have that the orbifold $Q$ admits a smooth structure 
if and only if each  $3$-dimensional link bounds a smooth contractible  $4$-manifold. 
\end{theorem}

The proof is based on the fact that every homology sphere of dimension $\neq 3$ has a smooth structure in which it smoothly bounds a contractible manifold (cf.~Lemma~\ref{l:hsphere} below).

This paper concerns the uniqueness question:  if two Coxeter orbifolds are combinatorially equivalent, are they isomorphic (as orbifolds)?  In the topological category the answer is yes.  This is a consequence of the facts that the Poincar\'e Conjecture and the topological h-cobordism Theorem hold in all dimensions.

In Section~\ref{s:unique} (Definition~\ref{d:4}) we define what it means for two combinatorially equivalent, smooth Coxeter orbifolds to have ``diffeomorphic $4$-dimensional faces.''  Our goal is  the following theorem.

\begin{theorem}\label{t:diffI}
Suppose $Q$ and $Q'$ are combinatorially equivalent, smooth Coxeter orbifolds with diffeomorphic $4$-dimensional faces. Then $Q$ and $Q'$ are diffeomorphic.
\end{theorem}

\begin{corollary}\label{cor:diffIII} \textup{(Wiemeler \cite{wiemeler}).}
Combinatorially equivalent Coxeter orbifolds of type  (III) are diffeomorphic.
\end{corollary}

On the other hand, according to the following theorem, without the hypothesis of diffeomorphic $4$-dimensional faces, Theorem~\ref{t:diffI} is false. As we will explain in Section~\ref{s:unique}, this is a consequence of a result of Akbulut \cite{ak}.
 
\begin{theorem}\label{t:notdiff}
In each dimension $n\ge 4$, there are smooth Coxeter orbifolds that are combinatorially equivalent but not diffeomorphic. Moreover, these orbifolds can be chosen to be aspherical.
\end{theorem}

The questions we are dealing with in Theorems~\ref{t:annals}, \ref{t:diffI} and \ref{t:notdiff} have nothing to do with the labelings of the codimension two strata of the reflection type orbifold, rather we are only concerned with its underlying structure as a manifold with corners -- the Coxeter group is extraneous.  For this reason,  we will reformulate  our results in what follows in terms of ``resolutions'' of cones on generalized homology spheres by ``contractible manifolds with contractible faces.'' 

Sections~\ref{s:reflec} and \ref{s:torus} give applications of Theorems~\ref{t:diffI} and \ref{t:notdiff} to the theory of reflection groups and to the theory of locally standard torus actions. Mikiya Masuda suggested that I add to the original version of this paper some details in the section on torus actions concerning smoothness questions.  On the basis of his very helpful comments, I have added Propositions~\ref{p:moment}, \ref{p:locally} and \ref{p:toric} which address questions of equivariant diffeomorphism versus equivariant homeomorphism.

After completing the first version of this paper, I learned of the recent work \cite{wiemeler} by M. Wiemeler on torus manifolds, with which this paper has substantial overlap.  In particular, Wiemeler proves the result for simple polytopes (Corollary~\ref{cor:diffIII}), as well as its application in Section~\ref{s:torus} to quasitoric manifolds of type (III)  (Theorem~\ref{t:masuda}).  

\section{Existence}\label{prelim}
\paragraph{Manifolds with faces.}
A topological $n$-manifold with boundary is a \emph{smooth manifold with corners} if it is locally differentiably modeled on the simplicial cone,  $[0,\infty)^n\subset \rr^n$. If $(x_1,\dots , x_n)\!:\!U\to [0,\infty)^n$ are local coordinates in a neighborhood of a point $x$, then $c(x)$, the number $c(x)$ of indices $i$ with $x_i=0$, is independent of the chart.  A component of $\{x\mid c(x)=m\}$ is a \emph{stratum} of {codimension $m$}.   One also can define the notion of a \emph{topological manifold with corners} by requiring the overlap maps to be strata-preserving (cf.~\cite[p.~180]{dbook}).  A manifold with corners is \emph{nice} if each stratum of codimension two is contained in the closure of exactly two strata of codimension one.  Niceness implies that the closure of a codimension $k$-stratum is also a manifold with corners.    

A nice manifold with corners $P$ will be called a \emph{manifold with faces}.  A \emph{face} is a closed stratum.  The \emph{$k$-skeleton} of $P$ is the union of faces of dimension $\le k$.  Two manifolds with faces, $P$ and $P'$, are \emph{topologically isomorphic} if there is a strata-preserving homeomorphism $P\to P'$.  (Of course, a diffeomorphism of smooth manifolds with corners is automatically strata-preserving.) 
Parallel to our earlier definitions in the Introduction, consider the following conditions on a manifold with faces $P$:
\begin{enumerateI}
\item
Each stratum of $P$ is a compact contractible manifold.
\item
Each stratum of $P$ is homeomorphic to a disk.  
\item
$P$ is isomorphic to a simple convex polytope.
\end{enumerateI}
In case (I), $P$ is a \emph{contractible manifold with contractible faces}.  In case (II), it is a \emph{cell with cellular faces}.  If we replace the word ``contractible'' by ``acyclic,'' we get the notion of an \emph{acyclic manifold with acyclic faces}.  

One defines $N(P)$, the \emph{nerve of $P$}, as before: it is the simplical complex with vertex set the set of codimension one faces and with a $k-1$-simplex for each face of codimension $k$.   
Two manifolds with  faces, $P$ and $P'$, are  \emph{combinatorially equivalent} if there is a simplicial isomorphismm $N(P)\to N(P')$. If $P$ is an acyclic $n$-manifold with acyclic faces, then $N(P)$ is a $\GHS^{n-1}$. 

\begin{example}
A simple convex polytope is naturally a smooth cell with cellular faces.  More generally, if $N$ is a triangulation of $S^{n-1}$ and if the link of each $k$-simplex is homeomorphic to $S^{n-k-2}$ (e.g., if the triangulation is PL), then $\cone (N)$ has a  dual cell structure giving it the structure of a (topological) cell with cellular faces.
\end{example}

\paragraph{Resolutions.}
Suppose $N$ is a $\GHS^{n-1}$.  A \emph{resolution} of $\cone (N)$ is a contractible manifold with contractible faces $P$ such that $N(P)=N$.  To define \emph{acyclic resolution} replace the word ``contractible'' by ``acyclic.''  A resolution is \emph{smooth} if $P$ is a smooth manifold with corners. (The use of the term ``resolution'' will be discussed in Remark~\ref{r:resolution} below.)

The following existence result essentially was proved in \cite{d83}.  It implies Theorem~\ref{t:annals} of the Introduction.
\begin{theorem}\label{t:resolution}\textup{(cf.~\cite[Thm.~17.2]{d83}).} Suppose $N$ is a $\GHS^{n-1}$.  Then the following statements are true.
\begin{enumerate1}
\item
$\cone (N)$ has a topological resolution. 
\item
$\cone (N)$ has a smooth acyclic resolution. 
\item
$\cone (N)$ has a smooth resolution 
if and only if each of its $3$-dimensional links bounds a smooth (or PL) contractible $4$-manifold.
\end{enumerate1}
\end{theorem}

We will recall the proof of this theorem at the end of this section since it is based on the same obstruction theory which we will need for the uniqueness result.

\paragraph{Homology spheres and homotopy spheres.}
Let $\gT_k$ be the Kervaire-Milnor group of oriented diffeomorphism classses of smooth structures on $S^k$ (cf.~\cite{km}).  Let $\gT^H_k$ be the group of homology cobordism classes of smooth homology $k$-spheres.  
The next lemma  is well-known.
\begin{lemma}\label{l:hsphere}
	\begin{enumerate1}
	\item
	Any homology $k$-sphere is the boundary of  a topological contractible $(k+1)$-manifold.
	\item
	For $k\neq 3$, any PL homology $k$-sphere is the boundary of a PL contractible $(k+1)$-manifold.
	\item
	Suppose $M^k$ is a smooth homology $k$-sphere.  For $k\geq 4$, there is a smooth homotopy $k$-sphere $\gS^k$, possibly with an 		exotic smooth structure, such that $M^k\#\gS^k$ bounds a smooth contractible $(k+1)$ manifold. (Here $\#$ is used for the connected 		sum operation.)
	\end{enumerate1}
\end{lemma}

Statement (2) is not true for every homology $3$-sphere  because of Rohklin's Theorem.

\begin{proof}[Sketch of proof of Lemma~\ref{l:hsphere}]
We note that special arguments are needed to prove (1) when $k=3$ or $4$.
We first prove (3). Suppose $k> 4$ and that $M^k$ is smooth.  As in \cite[Theorem 5.6]{hh} 
one can add $1$- and $2$-handles to $M^k\times [0,1]$ to get a simply connected, homology cobordism $W$ from $M^k$ to a smooth homotopy $k$-sphere.  If $\gS^k$ denotes this homotopy sphere with orientation reversed, then $M^k\#\gS^k$ bounds a contractible manifold.  When $k=4$,  one argues that $M^4$ bounds a framed manifold $W^5$ (since it represents $0$ in framed bordism) and then that one can do surgery to make $W^5$ contractible.  This proves (3). Since every PL homology sphere is smoothable, it also proves (2).

As for (1), when $k=4$, it is not known if $M^4$ admits a PL structure.  However, since we can do simply connected topological surgery, we can add $1$- and $2$-handles to $M^4\times [0,1]$ as before to get a simply connected homology cobordism $W$ from $M^4$ to a homotopy $4$-sphere, and by Freedman's proof of the Poincar\'e Conjecture in \cite{freedman} this is homeomorphic to the standard $S^4$.  Gluing on $D^5$ gives the desired contractible $5$-manifold.  When $k=3$, Freedman  \cite{freedman} proved that (1) holds  in the topological category.
\end{proof}

\begin{remark}\label{r:resolution} (\emph{On resolutions}). 
A \emph{polyhedral homology $m$-manifold} means a simplicial complex $N$ such that the link of any $k$-simplex has the same homology as $S^{m-k-1}$.  The ``dual cell''  of such a $k$-simplex in $N$ is the cone on its link, i.e., it is a contractible polyhedral homology $(m-k)$-manifold with boundary.  General  resolutions for polyhedral homology manifolds (rather than just for generalized homology spheres) were considered by Cohen \cite{cohen} and Sullivan \cite{sullivan} in the early 1970\,s.  They developed an obstruction theory for finding a resolution of $N$ by a manifold by replacing each dual cell by an acyclic manifold with boundary.  One proceeds by induction on the dimension of the dual cell.  By Lemma~\ref{l:hsphere}~(2), the only obstruction to finding a PL acyclic resolution of a polyhedral homology manifold $N$ lies in  $H^4(N;\gTh)$.  A version of this obstruction theory is used in the proof of Theorem~\ref{t:resolution} which we sketch  below.  In the topological category,  it follows from Lemma~\ref{l:hsphere}~(1) that contractible resolutions always exist.
\end{remark}

\begin{proof}[Sketch of proof of Theorem~\ref{t:resolution}]
We build the faces of $P$ by induction on dimension.  Start with a $0$-dimensional face for each $(n-1)$-simplex of $N$.  Assume by induction that we have constucted a face $F_\gt$ for each simplex $\gt$ of codimension $<k$ and let $\gs$ be a simplex of codimension $k$.  Define 
\(
\bo F_\gs :=\bigcup _{\gt>\gs} F_\gt
\)
where the union is over  all $\gt>\gs$ of codimension $(k-1)$ in $N$.  Then $\bo F_\gs$ is a homology $(k-1)$-sphere and the problem is to fill it in with a $k$-dimensional contractible face. By Lemma~\ref{l:hsphere}~(1) we can always do this in the topological category; hence, statement (1).

To get a smooth resolution one encounters a problem when trying to construct the $4$-dimensional faces. If each face of dimension $\le 3$ is a disk, then the boundary of a potential $4$-dimensional face is dual to the corresponding $3$-dimensional link.  So, the condition that this homology $3$-sphere smoothly bounds a contractible manifold is sufficient for constructing the $4$-dimensional face.  
\emph{A priori},  this condition might not be necessary since the homology $3$-sphere would be indeterminate if  $3$-dimensional faces could be  fake $3$-disks.  However, since the $3$-dimensional Poincar\'e Conjecture is true, fake $3$-disks do not exist. Therefore, a smooth  $4$-dimensional face exists if and only if the homology $3$-sphere bounds a contractible $4$-manifold. For $k\ge 4$, the obstruction to filling in the smooth $(k+1)$-dimensional strata lies in $H^k(\cone(N);\gT_{k-1})$.  This group is $0$ since $\cone (N)$ is acyclic.  This proves (3).

What about (2)?  As explained in \cite[pp.\,322--323]{d83}, there is an obstruction cocycle to filling in the boundaries of potential $4$-dimensional strata with smooth acyclic $4$-manifolds.  This cocycle takes values in $\gTh$.  However, this cocycle is indeterminate, since we can alter our construction of $3$-dimensional faces by taking connected sum with elements of $\gTh$.  In other words, the cocycle can varied by a coboundary giving us a well-defined obstruction in $H^4(\cone(N);\gTh) = 0$.  After filling in the $4$-dimensional faces, we can continue as in the proof of (3) to fill in the faces of higher dimension.
\end{proof}

\paragraph{Uniqueness in the topological category.}  The next lemma follows from the Poincar\'e Conjecture and the h-cobordism Theorem.
\begin{lemma}\label{l:topcat}
Suppose $C$ and $C'$ are compact contractible $k$-manifolds and that $\gf:\bo C\to \bo C'$ is a homeomorphism,  then $\gf$ extends to a homeomorphism $\Phi:C\to C'$.
\end{lemma}
Suppose $P$ and $P'$ are topological contractible $n$-manifolds with contractible faces and $N(P)=N(P')=N$.  It follows from Lemma~\ref{l:topcat} is that $P$ and $P'$ are toplogically isomorphic as manifolds with corners.  In other words, \emph{the topological resolution $\cone(N)$ is unique}.

\section{More facts about homology spheres and contractible manifolds}\label{s:basic}

\paragraph{Pseudo-isotopies on $S^n$.}  Given a smooth manifolds $M$ and $M'$,  denote by $\diff(M,M')$  the set of diffeomorphisms from $M$ to $M'$.  Let $\diff(M):=\diff(M,M)$, be the topological group of self-diffeomorphisms of $M$. If $M$ is oriented, $\diff^+(M)$ denotes the subgroup of orientation-preserving self-diffeomorphisms.  If $M$ is a manifold with boundary, then $\diff_\partial(M)$ is the subgroup of self-diffeomorphisms which are the identity on $\partial M$.

Let $\gGen$ denote the group of pseudo-isotopy classes of orientation preserving, self diffeomorphisms of $S^n$.   Let $\gi:\diff(D^{n+1})\to \diff(S^n)$ be the natural homomorphism which takes $\gf\in \diff(D^{n+1})$ to $\gf\vert_{S^n}$.  Here are three other equivalent definitions of $\gGen$:
	\begin{enumeratei}
	\item
	$\gGen= \diff(S^n)/\gi(\diff(D^{n+1}))$.  In other words, a diffeomorphism $\gf\in \diff(S^n)$ represents $0$ in $\gGen$ if and only if it 		extends to a self-diffeomorphism of $D^{n+1}$.
	\item
	$\gGen=\pi_0(\diff^+(S^n)$, the group of isotopy classes of orientation preserving diffeomorphisms of $S^n$.
	\item
	For $n+1\neq 4$, $\gGen=\gT_{n+1}$, the abelian group (under connected sum) of oriented diffeomorphism classes of smooth structures 	on $S^{n+1}$.  (Although $\gT_4$ is unknown, $\gG_4=0$.)
	\end{enumeratei}
Version (i) is obvious.  As for (ii), Cerf \cite{cerf1} proved that $\gG_4=0$, and  then in \cite{cerf2} that for simply connected manifolds of dimension $\ge 5$, pseudo-isotopic diffeomorphisms are isotopic.  Hence, for $n\ge 5$, $\gGen=\pi_0(\diff^+(S^n))$. It follows that $\gGen$ is a finite abelian group.

The next lemma is well-known and not difficult to prove.

\begin{lemma}\label{l:easy}
Let $f\in \diff^+(S^n)$.  Let $D'$ be a closed, smooth $n$-disk in $S^n$ and let $D^n$ be the closure of complementary $n$-disk.   Then $f$ is isotopic to a diffeomorphism $f'\in \diff^+(S^n)$ such that $f'(D')=D'$ and $f'\vert_{D'}$ is the identity.  Consequently, any element of $\gGen$ can be represented  by a diffeomorphism in $\diff_\partial (D^n)$.
\end{lemma}

\paragraph{Contractible manifolds and homology spheres.} The next two lemmas are well-known consequences of the smooth h-cobordism Theorem.
\begin{lemma}\label{l:homology}\textup{(cf.~\cite[Theorem on p.\,183]{levine}).}  Suppose $C$ and $C'$ are smooth, compact contractible $(n+1)$-manifolds with homeomorphic boundaries. If $n\neq 3$ or $4$, then $C$ and $C'$ are diffeomorphic.  In particular, if $n\neq 4$, $\bo C$ and $\bo C'$ are diffeomorphic.
\end{lemma}

\begin{proof}
Suppose $\gf:\bo C\to \bo C'$ is a homeomorphism.  Then $C\cup_\gf C'$ is homeomorphic to $S^{n+1}$.  Hence, it has a smooth structure in which it is diffeomorphic to $S^{n+1}$.  The union of a collared neighborhood of $\bo C$ in $C$ with a collared neighborhood of $\bo C'$ in $C'$,  gives a smooth h-cobordism from $\bo C$ to $\bo C'$; hence, for $n\neq 4$, they are diffeomorphic.  Filling in $C\cup_\gf C'$ with $D^{n+2}$ we get a smooth h-cobordism from $C$ to $C'$ which restricts to the given one from $\bo C$ to $\bo C'$.
\end{proof}

\begin{lemma}\label{l:simplytrans}  Let $M$ be a homology $n$-sphere and let $\gT_n(M)$ be the set of oriented diffeomorphism classes of smooth structures on $M$.  Suppose $n>4$.  Then the action of $\gT_n$ on $\gT_n(M)$ by connected sum is simply transitive.  In particular, if $M_0$ denotes the ``standard'' smooth structure on $M$ induced from the contractible manifold which it bounds,  then there is a bijection $\gT_n\to \gT_n (M)$ defined by $\gS\mapsto M_0\#\gS$.
\end{lemma}

\begin{proof}
The induced action of $\gT_n$ on $\gT^H_n$ is simply transitive (cf.~\cite[p.~183]{levine}).
\end{proof}

Suppose that $M$ and $M'$  are  smooth, closed $n$-manifolds, that $D^n\subset M$ and $D^n\subset M'$ are embedded disks and that $g\in \diff(M,M')$ is such that $g\vert_{D^n}=\id_{D^n}$. Given $f\in\diff_\partial (D^n)$, 
let $\hat{f}:M\to M'$ to be the diffeomorphism defined by extending $f$ via $g\vert_{M-D^n}$. When $M'=M$ and $g=\id_M$, this gives a homomorphism $f\mapsto \hat{f}$ from $\diff_\partial (D^n)$ to $\diff (M)$ which descends to a homomorphism $\gl:\gGen\to  \pi_0 (\diff(M))$.

Next suppose $A$, $A'$ are smooth structures on a  compact, acyclic $(n+1)$-manifold with boundary.  We want to define left inverse for the action of $\gGen$ on $\pi_0(\diff (\bo A,\bo A'))$.
Suppose $\gf:\partial A\to \partial A'$ is a diffeomorphism.  Then $A\cup_\gf A'$ is a smooth homology $(n+1)$-sphere.  For $n+1> 4$, we get a well-defined element $\gamma(\gf)$ in  $\gT^H\none=\gT\none=\gGen$,
	\begin{equation}\label{e:gamma}
	\gamma(\gf)=[A\cup_\gf A'],
	\end{equation}
where $[A\cup_\gf A']$ denotes the class of $A\cup_\gf A'$ in $\gT^H\none$.  

From its definition, $\gamma$ is clearly a left inverse to the action of $\gGen$ defined by $\gl$.  Thus, if $f\in \diff_\bo(D^n)$, $\gl(f)=\hat{f}:\bo A\to \bo A'$, then $\gamma(\gl(f))$ is the class of $f$ in $\gGen$.

\begin{remark}\label{r:contractible}
In dimension $3$, we have $\gG_3=\gT_3=0$, and $\gT^H_3\neq 0$.  In dimension $4$, $\gG_4=\gT^H_4=0$, and $\gT_4$ is unknown.  So, when $n+1=3$ or $4$, if $A$ and $A'$ are contractible, then $A\cup _\gf A'$ bounds a contractible manifold and we can require $\gamma(\gf)$ to lie in $\gGen$ (which is $=0$).
\end{remark}

We  will sometimes identify $A\cup _\gf A'$ with the smooth homology sphere:
	\begin{equation}\label{e:hsphere}
	\gS(\gf):= A \cup M_\gf \cup A',
	\end{equation}
where $M_\gf$ is the mapping cylinder of the diffeomorphism $\gf:\partial A\to \partial A'$.  

Since $A$ is acyclic, so is $A\times [0,1]$.  Hence, $\gS(\id)=\partial (A\times [0,1])$ represents $0$ in $\gT^H\none$. 
So,  if $\gf$ extends to a diffeomorphism $\Phi:A\to A'$, we get a diffeomorphism from $A\cup_{\id} A$  to $A\cup_\gf A'$.  In other words, if $\gf$ extends to a diffeomorphism, then $\gamma(\gf)=0$.  For $n+1>4$ and for $A$ and $A'$ contractible, the converse  is a consequence of the h-cobordism Theorem as we show in the next lemma.

\begin{lemma}\label{l:cmfld}
Suppose $C$ and $C'$ are smooth, compact contractible $(n+1)$-manifolds and that $\gf:\bo C\to \bo C'$ is a diffeomorphism.
	\begin{enumerate1}
	\item
	$C\cup_\gf C'$ smoothly bounds a contractible $(n+2)$-manifold if and only if $\gamma(\gf)=0$.
	\item
	For $n+1\neq 4$, $\gf$ extends to a diffeomorphism $\Phi:C\to C'$ if and only if $\gamma(\gf)=0$.
	\end{enumerate1}
\end{lemma}

\begin{proof}
(1) When $n+1=3$, the $3$-dimensional Poincar\'e Conjecture (cf.~\cite{perelman}) implies that $C\cup_\gf C' = S^3=\bo D^4$.  For $n+1=4$, by Lemma~\ref{l:hsphere}~(3), any smooth structure on $S^4$ bounds a contractible $5$-manifold.  For $n+1>4$, by Lemma~\ref{l:simplytrans}, $\gamma(\gf)=0$ if and only if $C\cup_\gf C'$ bounds a contractible manifold.

(2) As explained in the paragraph preceding this lemma, if $\gf$ extends to $\Phi$, then $\gamma(\gf)=0$.  Conversely, suppose $\gamma(\gf)=0$.  Then, since $\gS(\gf)=S^{n+1}$, we can fill in $\gS(\gf)$ with an $(n+2)$-disk to obtain an h-cobordism $W$ from $C$ to $C'$.  By the h-cobordism Theorem (which is true for $n+1\neq 4$), there is a diffeomorphism $\Phi\,:\,C\times [0,1] \to W$ such that $\Phi(C\times 0)=C$, $\Phi(C\times 1) = C'$,  $\Phi(\partial C\times [0,1]) =M_\gf$ and moreover, $\Phi\vert_{C\times 0} = \id$ and $\Phi\vert _{\partial C\times [0,1]}$ is the natural identification onto $M_\gf$.  Thus, $\Phi\vert_{\partial C\times 1} =\gf$ and $\Phi\vert_{C\times 1}$ is the desired extension.
\end{proof}

In the next section we will need the following lemma to establish that an obstruction cochain is a cocycle.

\begin{lemma}\label{l:holes}
Let $E$ be a smooth, compact contractible $(n+1)$-manifold and let
$F_1,\cdots , F_m$ be disjoint compact contractible submanifolds of codimension $0$ in $\bo E$.  Let $X^n=\bo E-\bigcup _{i=1}^{m} F^\circ_i $, 
where $F^\circ_i$ denotes the interior of $F_i$.  
Suppose $f:X^n\to X^n$ is a diffeomorphism which takes each $\bo F_i$ to itself.  Let $\gamma_i\in \gGen$ be the obstruction to extending $f\vert_{\bo F_i}$ to $F_i$.  Then $\gamma_1 +\cdots +\gamma_m =0$.
\end{lemma}

\begin{proof}
Choose disjoint embedded paths (``edges'') connecting the $\bo F_i$ so that the union of the edges with the $F_i$ is simply connected (giving a ``tree of contractible $n$-manifolds'' in $\bo E$). We can  arrange that $f$ restricts to the identity map on a tubular neighborhood of each edge.  Let $A$ denote the complement of these neighborhoods of edges in $X^n$.  In other words, $A$ is the complement of a regular neighborhood of the tree of contractible $n$-manifolds.   It follows that $A$ is a compact acyclic $n$-manifold, that $\bo A$ is the connected sum, $\bo F_1\#\cdots \#\bo F_m$ and that $f$ takes $A$ to itself.  The restriction of $f$ to $\bo A$ represents the element $\gamma_1 +\cdots +\gamma_m \in \gGen$.  Since $f$ extends over $A$, by paragraph preceding Lemma~\ref{l:cmfld}, this element is $0$.
\end{proof}

\begin{remark}\label{r:akb}
For $n+1=4$, Lemma~\ref{l:cmfld}~(2) is unknown even when $C$ and $C'$ are both homeomorphic to $D^4$.  (This is equivalent to the smooth $4$-dimensional Poincar\'e Conjecture.)  When $\bo C$ and $\bo C'$ are allowed to be homology $3$-spheres with nontrivial fundamental groups, 
Lemma~\ref{l:cmfld}~(2) does not hold.  Indeed, Akbulut \cite{ak} has shown that there exist compact contractible smooth $4$-manifolds $C$ and $C'$ with $\partial C= \partial C'$, such that $C$ and $C'$ are not diffeomorphic rel $\partial C$ (cf.~Theorem~\ref{t:akb} below).
\end{remark}

\section{Uniqueness}\label{s:unique}
A $k$-dimensional face $F$ of an $n$-dimensional manifold with faces has a tubular neighborhood of the form $F\times [0,\infty)^{n-k}$, where $[0,\infty)^{n-k} $  denotes the standard simplicial cone in $\rr^{n-k}$.  This is an easy consequence of the Collared Neighborhood Theorem for manifolds with boundary.  Given a smooth $n$-dimensional manifold with faces $P$, let  $\bohat P$ denote its topological boundary.  (N.B. We write $\bohat P$ instead of $\bo P$ to indicate that the corners have not been rounded.) 
The tubular neighborhoods of the faces can be fit together compatibly to give a neighborhood $e(P)$ of $\bohat P$.  If $P$ is smooth,  we can push $\bohat P$ into the interior of $e(P)$ to get a smooth manifold $\bo P$ ($\bo P$ is $\bohat P$ with  corners rounded).  The manifold $\bo P$ separates $P$ into two pieces, one is a neighborhood of $\bohat P$ and the other is homeomorphic to $P$.  It should not cause confusion to continue to denote the second piece by $P$.

Suppose that  $N$ is a $\GHS^{n-1}$ and that $P$, $P'$ are smooth contractible manifolds with contractible faces with $N(P)=N(P')=N$.  For each $\gs \in N$, let $F_\gs$ and $F'_\gs$ be the corresponding faces of $P$ and $P'$. We will try to construct a diffeomorphism $\gf:P\to P'$ one face at a time.  Since the faces of dimension $\le 3$ are cells, we can define $\gf$ on a neighborhood of the $3$-skeleton of $P$. It is unique up to isotopy.  Suppose $\gs$ is a  simplex of codimension $4$.  Then we have a diffeomorphism $\gf:\bo F_\gs\to \bo F'_\gs$ (which we can assume is the identity).  The problem is to extend it.  According to Remark~\ref{r:akb}, Lemma~\ref{l:cmfld}~(2) does not apply. This leads us to the following.

\begin{definition}\label{d:4}
Suppose, as above,  $P$ and $P'$ are smooth contractible manifolds with contractible faces and $N(P)=N(P')=N$.  Then $P$ and $P'$ have \emph{diffeomorphic $4$-dimensional faces} if for each simplex $\gs$ of codimension $4$ in $N$, $F_\gs$ and $F'_\gs$ are diffeomorphic rel boundary.
\end{definition}

Our goal is to prove the following, which is equivalent to Theorem~\ref{t:diffI}.

\begin{theorem}\label{t:main}
Suppose  $P$ and $P'$ are smooth contractible manifolds with contractible faces and $N(P)=N(P')=N$.  Then $P$ and $P'$ are diffeomorphic if and only if they have diffeomorphic $4$-dimensional faces.
\end{theorem}

Once we have constructed a $\gf$ taking a $(k-1)$-face $F$ to the corresponding $(k-1)$-face $F'$ of $P'$, we extend it to a map of tubular neighborhoods, $F\times [0,\infty)^{n-k} \to F' \times [0,\infty)^{n-k}$, which is linear on the $[0,\infty)^{n-k}$ factors.  So, suppose $\gf$ has been defined on the $(k-1)$-skeleton of $P$ and that we want to extend it to a map on a $k$-face, $F\to F'$. Since we have extended to tubular neighborhoods,  we have a diffeomorphism $\gf:e(F)\to e(F')$ restricting to a diffeomorphism $\bo F\to \bo F'$.  The problem is to extend this to a diffeomorphism $F\to F'$.  There might be an obstruction. 

\begin{proof}[Proof of Theorem~\ref{t:main}]
The faces of $P$ and $P'$ of dimension $\le 2$ are standard cells.  Since the $3$-dimensional Poincar\'e Conjecture is true (cf. \cite{perelman}, \cite{mt}), the $3$-dimensional faces are also standard. It follows that we can define $\gf$ on a neighborhood of the $3$-skeleton and this definition is unique up to isotopy.  The hypothesis that the $4$-dimensional faces of $P$ and $P'$ are diffeomorphic implies that  $\gf$ can be extended over a  neighborhood of the $4$-skeleton.

Suppose, by induction, that we have constructed $\gf$ on a tubular neighborhood of the $(k-1)$-skeleton, with $k-1\ge 4$.  One needs to choose orientations for the faces of $P$ in order to define the cellular cochains on $P$.  Let $F$ be an oriented $k$-face.  Then $\gf\vert_{\bo F}: \bo F\to \bo F'$ is defined.  As in the paragraph preceding Lemma~\ref{l:cmfld} there is an element $c(F)\in \gGekko$ defined by
	\[
	c(F)=\gamma(\gf)=[F\cup_\gf F'] \in \gT\kone=\gGekko.
	\]
By Lemma~\ref{l:cmfld}~(2), $\gf\vert_{\bo F}$ extends across $F$ if and only if $c(F)$ vanishes.  The assignment $F\mapsto c(F)$ is a cellular cochain $c\in C^k(P;\gG^k)$. It  is the obstruction to extending $\gf$ across the $k$-skeleton.  We have to check two  points in order for this  obstruction theory to work:
	\begin{enumerate1}
	\item
	We are free to vary the construction of $\gf$ on the $(k-1)$-skeleton.  If we fix $\gf$ on the $(k-2)$-skeleton and vary its extension over 		the $(k-1)$-faces, then $c$ should change by a coboundary.
	\item 
	$c$ should be a cocycle.
	\end{enumerate1}
	
To check (1) suppose $G$ is an oriented $(k-1)$-face and $\gamma \in \gGek$.  Let $d_{G,\gamma}\in C^{k-1}(P;\gGek)$ be the cochain which assigns $\gamma$ to $G$ and $0$ to all other $(k-1)$-faces.  Next we want to alter $\gf$ on the $(k-1)$-skeleton by changing $\gf\vert _G$ by the element $\gamma$ (where $\gamma$ is thought of as a diffeomorphism of $D^{k-1}$ which is the identity on $\bo D^{k-1}$). Denote the new map on the $(k-1)$-skeleton by $\psi$.  So, we have two cochains $c_\gf$, $c_\psi \in C^k(P;\gGek)$.  How are they related?  Let $F$ be an oriented $k$-face such that $G$ occurs in $\bohat F$ with coefficient $\pm 1$.  Denote this coefficient by $[G\!:\!F]$ (and call it the \emph{incidence number}).  Then $c_\gf$ and $c_\psi$ are related by the formula
	\[
	c_\psi (F)=c_\gf (F) +[G\!:\!F] \,d_{G,\gamma}.
	\]
That is to say, $c_\psi = c_\gf + \gd(d_{G,\gamma})$, where $\gd:C^{k-1}(P;\gGek) \to C^k(P;\gGek)$ is the coboundary map.  This establishes (1).

As for (2) we want to check that $c$ ($=c_\gf$) is a cocycle.  Let $E$ be a $(k+1)$-face and let $F_1,\dots, F_m$ the $k$-faces of $\bohat E$, oriented by outward-pointing normals.  (This means that all incidence numbers $[F_i:E]$ are $=1$.)   Apply Lemma~\ref{l:holes} with $X$ a regular neighborhood in $E$ of the $(k-1)$-skeleton of $E$ and with $\gamma_i= c(F_i)$ to get:
	\[
	\gd(c)(E)=c(\bohat E)= \sum_{i=1}^m c(F_i) =0 .
	\]
Hence, $c$ is a cocycle.

We use (1) and (2) to complete the proof.  Suppose $c\in C^{k-1}(P;\gGek)$ is the obstruction cochain.
Since $c$ is a cocycle and  $P$ is acyclic, we have $c=\gd(d)$ for some $d\in C^{k-1}(P;\gGek)$.  If we use $-d$ to alter $\gf$ on the $(k-1)$-skeleton,  the new $c$ will be identically $0$.  Hence, we can extend over the $k$-skeleton.
\end{proof}

\paragraph{Nonuniqueness.}  In the language of resolutions, Theorem~\ref{t:notdiff} can be restated as follows.

\begin{theorem}\label{t:notdiff2}
For each $n\ge 4$ there is a generalized homology $(n-1)$-sphere $N$ with two distinct smooth resolutions.
\end{theorem}

This is more or less an immediate consequence of the following important result of Akbulut \cite{ak}.

\begin{theorem}\label{t:akb}\textup{(Akbulut \cite{ak}).}
There exist two smooth, compact contractible $4$-manifolds $C$ and $C'$ with the same boundary so that $C$ and $C'$ are not diffeomorphic rel boundary.
\end{theorem}

In Akbulut's construction $C$ is the Mazur manifold.  (This was the first example of a compact, contractible $4$-manifold whose boundary is a  homology sphere that is not simply connected.)

Here are a few remarks about the proof of \ref{t:notdiff2}.  First suppose $n=4$.  Let $N$ be a triangulation of the boundary of the Mazur manifold.  The dual cell structure gives a resolution of $N$.  We can get two different smooth resolutions $P$ and $P'$ of $\cone(N)$ by filling in $C$ and $C'$, respectively.  Since the diffeomorphism $\bo P \to \bo P'$ is isotopic to the identity, it cannot be extended to a diffeomorphism $P\to P'$.

For $n>4$, let $N$ be any $\GHS^{n-1}$ which has $\bo C$ as the link of a  simplex of codimension $4$.  For example, one could take $N$ to be the join $\bo C * S^{n-5}$ (or any subdivision of this).  By filling in each $4$-dimensional face with copy of $C$ or $C'$ one obtains distinct resolutions.

\section{Reflection groups}\label{s:reflec}
\paragraph{Smooth equivariant  rigidity of reflection groups.} Associated to a Coxeter orbifold $Q$, there is a Coxeter system $(W,S)$,  where the set of generators $S$ corresponds to the set of codimension one faces of $Q$.  $W$ is the orbifold fundamental group of $Q$.  If $L(W,S)$ denotes the nerve of the Coxeter system, then properness of the labeling means that $N(Q)\subset L(W,S)$.  If $L(W,S)=N(Q)$, then the universal cover $\wt{Q}$ of $Q$ is contractible (cf.~\cite{d83} or \cite{dbook}), i.e.,  $Q$ is aspherical.  If this is the case, then $\wt{Q}=\ue W$, the universal space for proper $W$-actions.  Since $L(W,S)$ is a $\GHS^{n-1}$, $W$ is a Coxeter group of type $\HM^n$ (cf.\ \cite[Chapter 10]{dbook}).

Now suppose $Q'$ is another aspherical Coxeter orbifold with associated Coxeter system $(W',S')$ and with $W\cong W'$.  By the universal properties of $\ue W$ and $\ue W'$, there is a $W$-homotopy equivalence $\wt{Q}\to \wt{Q}'$. It is proved in \cite{cd} that $L(W,S)$ and $L(W',S')$ are isomorphic, i.e., that $Q$ and $Q'$ are combinatorially equivalent. From this we get the rigidity theorem of Prassidis-Spieler in \cite{ps}: $\wt{Q}$ and $\wt{Q}'$ are $W$-equivariantly homeomorphic. When combined with Theorem~\ref{t:diffI}, this gives the following.

\begin{theorem}\label{t:diffrig}\textup{(cf.~Prassidis-Spieler \cite{ps}).}
Suppose $Q$ and $Q'$ are aspherical Coxeter orbifolds with isomorphic Coxeter groups.  Then $\wt{Q}$ and $\wt{Q}'$ are equivariantly diffeomorphic if and only if corresponding $4$-dimensional faces of $Q$ and $Q'$ are diffeomorphic. 
\end{theorem}

In other words, the equivariant, smooth version of the Borel Conjecture holds for aspherical Coxeter orbifolds if and only if their $4$-dimensional faces are diffeomorphic. Theorem~\ref{t:notdiff} yields the following.

\begin{theorem}\label{t:fake}
For each $n\ge 4$, there are $n$-dimensional Coxeter orbifolds $Q$ and $Q'$ which are combinatorially equivalent but not diffeomorphic (hence, $\wt{Q}$ and $\wt{Q}'$ are not equivariantly diffeomorphic).
\end{theorem}
There is no problem in arranging for $Q$ and $Q'$ to be aspherical.  For example, if $N$ is a $\GHS^{n-1}$  discussed in the last two paragraphs of the previous section, we can insure asphericity by replacing $N$ by its barycentric subdivision and labeling all edges $2$.

The following question remains open.
\begin{question}
Are there fake smooth closed aspherical $4$-manifolds?
\end{question}

If $Q$ and $Q'$ are aspherical $4$-dimensional orbifolds as in Theorem~\ref{t:fake}, then by passing to a subgroup of finite index in the Coxeter group, we obtain smooth aspherical $4$-manifolds $M$ and $M'$ with the same fundamental group.  They might or might not be diffeomorphic.  A variation of this would be to change only one chamber of $M$ from $C$ to $C'$.  It seems plausible that some such construction could yield homeomorphic but not diffeomorphic aspherical manifolds.

\section{Torus actions}\label{s:torus}
\paragraph{The moment angle manifold.}
There is a standard linear action of the $m$-torus $T^m$ on $\cc^m$.  The orbit space is $\rr_+^m$, where $\rr_+=[0,\infty)$.   The orbit map $p:\cc^m\to \rr_+^m$ can be defined by $p(z_1,\dots, z_m)=(|z_1|^2,\dots, |z_m|^2)$.

Suppose $P$ is an $n$-dimensional smooth manifold with faces with $m$ faces of codimension one, $F_1, \dots, F_m$.  Let $f:P\to \rr_+^m$ be a map such that the inverse image of the coordinate hyperplane $x_i=0$ is $F_i$.  Moreover, $f$ is required to be transverse to each of these hyperplanes.  Next we will construct a smooth $(n+m)$-manifold $Z_P$ with a smooth $T^m$-action and with orbit space $P$.  The \emph{moment-angle manifold} $Z_P$ corresponding to $P$ is defined by the pullback diagram:
\begin{equation}\label{e:pullback}
	\begin{CD}
	Z_P @>>> \cc^m\\
	@VVV @VVpV\\
	P@>f>> \rr_+^m
	\end{CD}
\end{equation}
In other words, $Z_P=f^*(\cc^m):=\{(x,z)\in P\times \cc^m \mid f(x)=p(z)\}$.  By the tranversality hypothesis,  $Z_P$ inherits the structure, as a subset of $P\times \cc^m$, of a smooth $(n+m)$-manifold.   Moreover, the $T^m$-action on the second factor induces a smooth $T^m$-action on $Z_P$.  As we will see in Proposition~\ref{p:moment} below, up to  an equivariant diffeomorphism inducing the identity on $P$, the moment angle polytope is independent of the choice of $f$.

The $T^m$-action on $Z_P$ is \emph{modeled on the standard representation} in the sense that its orbit types and normal representations occur among those of $T^m$ on $\cc^m$.  The principal orbits (i.e., the maximal orbits) are isomorphic to $T^m$.  The principal orbit bundle is trivial (since it is the pullback of the principal orbit bundle for $T^m$ on $\cc^m$).

\begin{Remarks}\label{r:pullback}
For $G$ a compact Lie group, the question of pulling back smooth  $G$-manifolds, which are modeled on certain types of representions,  from their linear models was the topic of my PhD thesis (cf.~\cite{d78} and \cite{d81}).  In these papers I was mainly concerned with cases such as $G=O(n)$, $U(n)$ or $Sp(n)$ where the linear model was a multiple of the standard representation.  The case of the standard $T^m$-action on $\cc^m$ is probably the simplest special case of the general theory developed in \cite{d81}.  Here are a few remarks concerning the general theory.

a) Suppose  a smooth $G$-manifold $M$ is modeled on a linear representation of $G$ on a vector space $V$.  An equivariant map between $G$-manifolds is \emph{isovariant} if it preserves isotropy subgroups and \emph{transverse isovariant} if its differential induces isomorphisms between normal representations to the strata.  A transverse isovariant map $F:M\to V$ induces a map $f:M/G\to V/G$ of orbit spaces and an equivariant diffeomorphism, $M\cong f^*(V)$, where $f^*(V)$ is defined as in \eqref{e:pullback}.  Thus,  $M$ is a pullback of its linear model if and only if it admits a transverse isovariant map to its linear model.

b)  For $B=M/G$, consider maps $B\to V/G$ which are strata-preserving and ``transverse"" to the strata in some obvious sense (for short, \emph{stratified maps}).  Let $f_0$ and $f_1$ be two maps from $B$ to $V/G$ which are homotopic through stratified maps.  It is a consequence of G. Schwarz's Covering Homotopy Theorem that $f_0^*(V)\cong f_1^*(V)$ (cf.~\cite{schwarz}).

c) Suppose $\ga$ is the ``normal orbit type'' corresponding to $(H, E)$ where $H$ is an isotropy subgroup and $E$ is the normal representation to the stratum of $G/H$ orbits.  The normal bundle to this stratum is a bundle over the corresponding stratum $B_\ga$ in the orbit space.  The fiber of this bundle over $B_\ga$ is $G\times_H E$.  After choosing a metric on the normal bundle, the structure group reduces to the group $S_\ga:=N_H(G\times O(E))/H$ (where $O(E)$ is the orthogonal group of $E$ and $N_H(G\times O(E))$ is the normalizer of $H$ in $G\times O(E)$).    Let $S_\prin$ be the structure group for the principal orbit bundle.  For some linear models, $S_\ga$ is a subgroup of $S_\prin$.  If the principal orbit bundle of $M$ is trivial, we get a map $B_\ga\to S_\prin/S_\ga$ called the \emph{$\ga$-twist invariant} of $M$.  If, as in the situations of interest in \cite{d81}, the $\ga$-stratum of $M/G$ is homotopy equivalent to $S_\prin/S_\ga$, the twist invariants can be used to provide a transverse isovariant map $M\to V$.
\end{Remarks}

\begin{proposition}\label{p:moment}
Suppose $Y^{m+n}$ is a smooth $T^m$-manifold, modeled on the standard representation and that the bundle of principal orbits is trivial.  Let $P =Y/T^m$ and let $F_i$ be the codimension one face of $P$ whose isotropy subgroup is the coordinate circle $T_i$  Then $Y$ is the pullback of the linear model via a stratified map $f:P\to \rr_+^m$ which takes $F_i$ into the hyperplane $x_i=0$. Hence, $Y$ is $T^m$-equivariantly diffeomorphic to the moment angle manifold $Z_P$ defined in \eqref{e:pullback}.
\end{proposition} 

\begin{proof}[Sketch of Proof]
In the case at hand, each $S_\ga=T^m = S_\prin$, so the range of each twist invariant is a point.  Since each stratum of $\rr_+^m$ is contractible, the twist invariants can be used to construct a map $Y\to \cc^m$ inducing $P\to \rr_+^m$ (cf.~Remark~\ref{r:pullback}~c)).  For the same reason, any two maps to $\rr_+^m$ are homotopic through stratified maps; so, by Remark~\ref{r:pullback}~b), any two pullbacks of $\cc^m$ are equivariantly diffeomorphic. Hence, $Y\cong Z_P$.
\end{proof}

\begin{corollary}
Suppose $Y$ and $Y'$ are smooth $T^m$-manifolds which are modeled on the standard representation and which have trivial bundles of principal orbits.   Then $Y$ is equivariantly diffeomorphic to  $Y'$ if and only if $Y/T^m$ and $Y'/T^m$ are diffeomorphic as manifolds with corners (via a strata-preserving diffeomorphism).
\end{corollary}

\paragraph{Quasitoric manifolds.}  A $T^n$-action on a manifold $M^{2n}$ is \emph{locally standard} if it is locally modeled on the standard representation $(\cc^n,T^n)$, up to automorphisms of $T^n$.  (N.B.\ Because we are allowing ourselves to vary the local actions by automorphisms of $T^n$, a locally standard action is usually \emph{not} modeled on the standard representation.  For example, in the standard representation only the coordinate circles $T_1,\dots, T_n$ occur as isotropy subgroups, while in a locally standard action it is possible to have more isotropy subgroups isomorphic to $S^1$.)
The orbit space $P$ of a locally standard action on $M^{2n}$ is a smooth $n$-manifold with corners.  In \cite{mp} Masuda and Panov impose the following additional requirement:
\medskip

($*$) $P$ is an acyclic manifold with acyclic faces.

\medskip

\noindent
(Such a $P$ is called a ``homology polytope'' in \cite{mp}.)  From now on we suppose  that $P$ satisfies ($*$).  $M^{2n}$ is a \emph{quasitoric manifold} if $P$ is a simple convex polytope (cf.~\cite{dj91}).  Let $F_1,\dots, F_m$  be the codimension one faces of $P$.  The isotropy subgroup at an interior point of $F_i$ is a subgroup $\gL_i$ isomorphic to $S^1$.  If $F_{i_1}\cap \cdots \cap F_{i_k} \neq \emptyset$, then $\gL_{i_1}, \dots ,\gL_{i_k}$ span a $k$-dimensional subtorus of $T^n$.  The subgroup $\gL_i$ is determined by a vector $\gl_i\in \Hom (S^1,T^m)\cong \zz^m$, well defined up to sign.  This defines an epimorphism $\gl=(\gl_1,\dots, \gl_m):T^m\to T^n$ called the \emph{characteristic function}.  We will sometimes also view $\gl$ as a homomorphism  $\zz^m \to \zz^n$.) It is observed in \cite{mp} that, as in \cite{dj91},  $M^{2n}$ is determined up to equivariant homeomorphism by $P$ and the characteristic function $\gl$.  In fact, as we will see in Proposition~\ref{p:locally} below, we can replace ``equivariant homeomorphism'' by  ``equivariant diffeomorphism.'' 

Put $T=T^n$ and let $M_T:=ET\times _T M$ be the Borel construction.  So, $M_T$ is a bundle over $BT$ with fiber $M$.  Masuda and Panov prove in \cite[Theorem 1]{mp} that when ($*$) holds, the cohomology of $M$ vanishes in odd degrees and that its cohomology is  genenerated by degree two classes.  It follows that the Serre spectral sequence for $M_T\to B_T$ degenerates at $E_2$ and that we have an isomorphism of $H^*(BT)$-modules:
\[
H^*(M_T)\cong H^*(BT) \otimes H^*(M).
\]
In particular, there is a short exact sequence,
	\begin{equation}\label{e:se}
	0\to H^2(BT)\to H^2(M_T)\to H^2(M)\to 0.
	\end{equation}

Given $P$, an acyclic manifold with acyclic faces, and a characteristic function $\gl:T^m\to T$, we construct a smooth $T$-manifold $M_{(P,\gl)}$ as follows.  Put $H=\Ker \gl$.  Then $H\cong T^{m-n}$ and $\gl$ induces an isomorphism $\ol{\gl}:T^m/H\to T$.  The subgroup $H$ acts freely on the moment angle manifold $Z_P$.  So, the quotient
	\begin{equation}\label{e:reduc}
	M_{(P,\gl)}=Z_P/H
	\end{equation}
is a smooth manifold with a smooth action of $T^m/H$.  After using $\ol{\gl}^{-1}$ to identify $T$ with $T^m/H$, we get a locally standard $T$-action on $M_{(P,\gl)}$ with orbit space $P$.  Moreover, the isotropy subgroup corresponding to the codimension one face $F_i$ is the circle $\gL_i$.

I am grateful to Mikiya Masuda for suggesting the proof of the following proposition in some e-mail correspondence.
\begin{proposition}\label{p:locally}
For $T=T^n$, suppose $M^{2n}$ has a locally standard, smooth $T$-action with orbit space $P$, where $P$ is an acyclic manifold with acyclic faces, with codimension one faces $F_1,\dots, F_m$.  Let $\gl:T^m\to T$ be its characteristic function and put $H=\Ker \gl$.  There is a smooth principal $H$-bundle $\pi: Y^{n+m} \to M^{2n}$ with $H_1(Y)=0$.  Moreover, there is a lift of the $T$-action to $Y$ giving a smooth $(T\times H)$-action on $Y$ and an isomorphism $L:T^m\to T\times H$ such that 
\begin{enumeratei}
\item
If $p_1:T\times H\to T$ denotes projection onto the first factor,  then $p_1\circ L=\gl:T^m\to T$ and  $L(H)=H$.
\item
After using $L^\minus$ to get a $T^m$-action on $Y$, the $T^m$-action is modeled on the standard representation $(\cc^m, T^m)$.
\item
$Y$ is $T^m$-equivariantly diffeomorphic to $Z_P$ and $M^{2n}$
is $T$-equivariantly diffeomorphic to $M_{(P,\gl)}$ defined by \eqref{e:reduc}.
\end{enumeratei}
\end{proposition}

We begin with some notation that will be used in the proof.  Principal $H$-bundles over $M$ are classified by homotopy classes of maps from $M$ to $BH$, i.e., by $[M,BH]$.  According to a  theorem of Hattori-Yoshida \cite{hy} (also cf.\ \cite{stewart}), given a principal $H$-bundle $Y\to M$, the $T$-action on $M$ lifts to a $(T\times H)$-action on $Y$ if and only if $Y$ is the pullback of a bundle over $M_T$.  Since $H\cong T^{n-m}$, $[M,BH]$ is the product of $(m-n)$ copies of $[M,BS^1]$, i.e., $[M,BH]\cong\bigoplus_{i=1}^{n-m} H^2(M)$.  By \eqref{e:se}, $H^2(M_T)\to H^2(M)$ is onto and hence, so is $[M_T,BH]\to [M,BH]$.  Therefore, any principal $H$-bundle over $M$ is the pullback of a principal $H$-bundle over $M_T$.  So, the Hattori-Yoshida Theorem  applies.  Let $L_i < (T\times H)$ be the isotropy subgroup at a point in the relative interior of $F_i$.  The group $T\times H$ is an $m$-dimensional torus.  To get a $T^m$-action on $Y$ which is modeled on the standard representation we need  the homomorphism $T^m\to T\times H$ which takes the coordinate circle $T_i$ to $L_i$ to be an isomorphism.  (Given a collection of $L_1,\dots, L_m$ of circle subgroups in $T\times H$, we say that the $L_i$ \emph{span} $T\times H$ if the homomorphism $T^m\to T\times H$ which takes $T_i$ to $L_i$ is an isomorphism.)

\begin{lemma}\label{l:span}
Suppose $M$ is as above and that $Y$ is a principal $H$-bundle over $M$.  Then the $L_i$ span $T\times H$ if and only if $H_1(Y)=0$.
\end{lemma}

\begin{proof}
Let $q:M\to P$ be the orbit map and $\pi:Y\to M$ the bundle projection.
Choose a point $p$ in the interior of $P$. Let $I_i$ be a line segment joining $p$ to a point in the relative interior of $F_i$.  Put $A=\bigcup I_i$, $U=q^\minus (A)$ and $V=\pi^{-1}(U)$.  Then $q^\minus(p)=T$ and $q^\minus (I_i)$ is a $D^2$-bundle over $T/\gL_i$.  We have $\pi_1(T)\cong \zz^n$ and $\pi_1(H)\cong \zz^{m-n}$.  Hence, $\pi_1(q^\minus(I_i))= \zz^n/\langle \gl_i\rangle$, where $\langle \gl_i\rangle$ is the $\zz$-submodule determined by $\gL_i$, and $\pi_1(U)=\zz^n/\sum \langle \gl_i\rangle$.  Since the $\langle \gl_i\rangle$ span $\zz^n$, $\pi_1(U)$ is trivial.  Similarly, $\pi^\minus (q^\minus(p))=T\times H$ and $\pi^\minus(q^\minus (I_i))$ is a $D^2$-bundle over $(T\times H)/L_i$.  Hence, $\pi_1(\pi^\minus(q^\minus(I_i)))= (\zz^n \oplus \zz^{m-n})/\langle l_i\rangle$ , where $\langle l_i\rangle$ is the $\zz$-submodule determined by $L_i$.  So, $\pi_1(V)\cong (\zz^n\oplus \zz^{m-n})/\sum \langle l_i\rangle$.  

A spectral sequence argument can be used to prove the following.
\begin{claim}
The inclusion $V\hookrightarrow Y$ induces an isomorphism $H_1(V)\cong H_1(Y)$.
\end{claim}
\noindent
Assuming this claim, we have $H_1(Y)\cong H_1(V)\cong (\zz^n\oplus \zz^{m-n})/\sum \langle l_i\rangle$ and since the quotient is trivial if and only if the $L_i$ span $T\times H$, we get the lemma.   

To prove the claim, note that $M_T$ and $U_T$ are the Davis-Januszkiewicz spaces for the simple polyhedral complexes $P$ and $A$, respectively (cf.\ \cite[Section 4]{dj91}).  By \cite[Theorem 4.8]{dj91}, $H^*(M_T)$ and $H^*(U_T)$ are the face rings of the simplicial complexes $N(P)$ and $N(A)$ which are dual to $P$ and $A$.  Such face rings are generated by $H^2$, and $H^2$ is free abelian on the vertex set of the simplicial complex.  Since $N(P)$ and $N(A)$ have the same vertex set (namely, $m$ points), the inclusion induces an isomorphism $H^2(M_T)\mapright{\cong} H^2(U_T)\cong \zz^m$ and hence, also an isomorphism on homology,  $H_2(U_T)\mapright{\cong} H_2(M_T)$.  Since $P$ is acyclic, it follows from \cite{mp} that $H_1(M)=0=H_1(U)$.  Comparing the Serre spectral sequences in homology for $U_T\to BT$ and $M_T\to BT$, we see that $H_2(U)\mapright{\cong} H_2(M)$.  Then comparing the spectral sequences for the principal $H$-bundles, $V\to U$ and $Y\to M$, we see that $H_2(V)\mapright {\cong} H_2(Y)$.  
Finally, if $d_2: E^2_{2,0}\to E^2_{0,1}$ denotes the $E^2$-differential, 
\begin{align}
H_1(V)&=\Coker (d_2:H_2(U)\to H_1 (H)\cong \zz^{m-n})\notag\\
H_1(Y)&=\Coker (d_2:H_2(M)\to H_1 (H)\cong \zz^{m-n}).\label{e:d2}
\end{align}
So, $H_1(V)\cong H_1(Y)$ establishing the claim and consequently, the lemma.
\end{proof}

\begin{proof}[Proof of Proposition~\ref{p:locally}]
The  characteristic class $c(Y)=(c_1, \dots ,c_{m-n})$ of a principal $H$-bundle $Y\to M$ lies in $\bigoplus_{i=1}^{m-n} H^2(M)$, where $H^2(M)\cong \zz^{m-n}$.  According to Lemma~\ref{l:span}, we want $H_1(Y)=0$.  To achieve this, choose $c_1, \dots, c_{m-n}$ to be a basis for $H^2(M)$.  We will then have that  $d_2:H_2(M)\to H_1(H)$ is onto and by \eqref{e:d2}, that $H_1(Y)=0$.  By Lemma~\ref{l:span}, the $L_i$ span $T\times H$.  Let $L:T^m\to T\times H$ be an isomorphism which sends $T_i$ to $L_i$.  From the definitions of $L_i$ and $\gL_i$, we see that $p_1:T\times H\to T$ takes $L_i$ to $\gL_i$; so, $p_1\circ L=\gl:T_m\to T$. Hence, $L(H)=L(\Ker \gl)=\Ker p_1 = H$.   So, Properties (i) and (ii) hold.  By Proposition~\ref{p:moment}, $Y$ is $T^m$-equivariantly diffeomorphic to $Z_P$.  Hence, $M=Y/H$ is $T$-equivariantly diffeomorphic to $M_{(P,\gl)}$, establishing Property (iii).
\end{proof}

By definition, the \emph{equivariant cohomology} of $M$ is the cohomology of $M_T$.  It is an algebra over $H^*(BT)$.  Masuda \cite{masuda} shows  that if  quasitoric manifolds have isomorphic equivariant cohomology as algebras over $H^*(BT)$, then they are equivariantly homeomorphic (\cite[Theorem 4]{masuda}).  Using Proposition~\ref{p:locally} and Corollary~\ref{cor:diffIII} we can upgrade Masuda's rigidity result from equivariant homeomorphism to equivariant diffeomorphism, as inTheorem~\ref{t:masuda} below.  (A different proof of this, using the theory of ``normal systems,'' is given by Wiemeler \cite{wiemeler}.)

\begin{theorem}\label{t:masuda}\textup{(cf.~Masuda \cite{masuda}, Wiemeler \cite[Corollary 5.7]{wiemeler}).}
Quasitoric manifolds have isomorphic equivariant cohomology as algebras over $H^*(BT)$ if and only if they are equivariantly diffeomorphic.
\end{theorem}

On the other hand, from Theorem~\ref{t:notdiff} we get the following.

\begin{theorem}\label{t:qt}
For each $n\ge 4$, there are locally standard $T^n$-manifolds over contractible manifolds with contractible faces (i.e., Coxeter orbifolds of type (I)) which are equivariantly homeomorphic but not equivariantly diffeomorphic. 
\end{theorem}

\paragraph{Toric varieties.} A nonsingular, compact toric variety is determined by a nonsingular, complete simplicial fan in $\rr^n$.  The intersection of such a fan with the unit sphere $S^{n-1}$  is a totally geodesic triangulation $N$ of $S^{n-1}$.   Such a geodesic triangulation of $S^{n-1}$ need not be simplicially isomorphic to the boundary complex of a simplicial polytope, e.g., see \cite[p.\,194]{ziegler}.  (We were not cognizant of this fact when we wrote \cite{dj91}.)  The rays of the fan are rational and determine a characteristic function $\gl$.  So, the fan determines a locally standard $T$-action on $M^{2n}$.  The orbit space $P^n= M^{2n}/T$ has the structure of a smooth manifold with faces.  Since $N$ is a PL triangulation of $S^{n-1}$, $P$ is a Coxeter orbifold of type (II).  When the toric variety is projective, one can use the moment map to identify $P^n$ with a simple polytope. In particular, the faces of $P$ are diffeomorphic to  disks.  As we will see below, this holds in general (it is not automatic since the smooth $4$-dimensional Poincar\'e Conjecture is not known.)

\begin{proposition}\label{p:toric}
Suppose $M^{2n}$ is nonsingular, compact toric variety and $P=M^{2n}/T$.  Then each face of $P$ is diffeomorphic to a disk.
\end{proposition}

\begin{proof}
Each top-dimensional simplex of $N$ corresponds to a $T$-action on $(D^2)^n$ with orbit space the $n$-cube, $[0,1]^n$.  The manifold with faces $P$ is obtained by gluing together these $n$-cubes via linear isomorphisms of faces. This defines a PL structure on $P$ so that  $\partial P$ is $N$ with its dual PL cell structure.  Therefore, each face of $P$ (and in particular, each $4$-dimensional face) is PL homeomorphic to a disk.  Theorem~\ref{t:diffI} completes the proof.
\end{proof}

\obeylines
Department of Mathematics 
The Ohio State University 
231 W. 18th Ave. 
Columbus, Ohio 43210 
{\tt davis.12@math.osu.edu}

\end{document}